\documentclass[12pt]{article}
\usepackage{amssymb}
\usepackage{amsthm}
\newtheorem{theorem}{Theorem}[section]
\newtheorem{definition}[theorem]{Definition}
\newtheorem{lemma}[theorem]{Lemma}

\newtheorem{remark}[theorem]{Remark}
\newtheorem{example}[theorem]{Example}

\title{Critical elements in algebras of numerical events}
\author{Dietmar~Dorninger and Helmut~L\"anger}
\date{}

\begin{document}

\footnotetext{Support of the research of the second author by the Austrian Science Fund (FWF), project 10.55776/PIN5424624, is gratefully acknowledged.}

\maketitle

\begin{abstract}
The probability $p(s)$ of the occurrence of an event pertaining to a physical system which is observed in different states $s$ determines a function $p$ from the set $S$ of states of the system to $[0,1]$. The function $p$ is called a numerical event or more precisely an $S$-probability. When appropriately structured, sets $P$ of numerical events form so-called algebras of $S$-probabilities, which are orthomodular posets that can serve as quantum logics. If one deals with a classical physical system, then $P$ will be a Boolean algebra. Starting with a supposed quantum logic or logics obtained by individual measurements, newly added numerical events may turn out to be crucial for the assumption of classicality or non-classicality of the physical system. We will call those numerical events critical and will study their impacts among various classes of algebras of numerical events. Moreover, we will consider the situation of enlarging $S$ by new states and will describe how elements of an algebra of $S$-probabilities contribute to the character of a logic when there exists a certain relation between them and the element that is newly added.
\end{abstract}

{\bf AMS Subject Classification:} 06C15, 06E99, 03G12, 81P10

{\bf Keywords:} Quantum logic, numerical event, probability of an event, state, orthomodular poset, Boolean algebra

\section{Introduction}

In quantum mechanics one can discriminate between classical mechanical phenomena and those that are not classical by means of the underlying quantum logic, which (from an algebraic point of view) will in most cases belong to the class of orthomodular posets. Within this framework classicality can be characterized by the fact that the underlying logic is a Boolean algebra.

Among other possibilities logics can be set up by means of so called numerical events: Let $S$ be a set of states of a physical system and $p(s)$ be the probability of the occurrence of an event (pertaining to a certain observable) when the system is in state $s\in S$, then the function $p\colon S\to[0,1]$ is called a {\em numerical event} or, more precisely, an {\em $S$-probability}.

Agreeing to use the symbols $0$ and $1$ as well for the integers $0$ and $1$ as for the constant functions mapping $S$ to $0$ and $1$, respectively, we give the following 

\begin{definition}\label{def1}
{\rm(\cite{BM91}} and {\rm\cite{BM93})} A set $P$ of $S$-probabilities comprising the constant functions $0$ and $1$ and ordered by the partial order $\le$ of real functions is called an {\em algebra of $S$-probabilities} or {\em algebra of numerical events}, if with $+$ and $-$ for the sum and difference of real functions, respectively, the following holds:   .
\begin{enumerate}
\item[{\rm(a)}] $0\in P$,
\item[{\rm(b)}] $p':=1-p\in P$ for all $p\in P$,
\item[{\rm(c)}] if $p,q,r\in P$ are pairwise orthogonal, i.e.\ $p\le q'$, $q\le r'$ and $r\le p'$, then $p+q+r\in P$.
\end{enumerate}
We denote algebras of $S$-probabilities by $(P,\le,0,1,{}',+)$ and for short only write $P$.
\end{definition}

For the orthogonality relation we will use the notation $\perp$ ($p\perp q$ means $p\leq q'$) and we call three pairwise orthogonal elements an {\em orthogonal triple}. Moreover, we observe that (a) -- (c) imply that if $p\perp q$ for $p,q\in P$ then $p+q\in P$ and $p+q$ is the supremum $p\vee q$ of $p$ and $q$ (see \cite{MT}). Further, we point out that $+$ is only defined for orthogonal elements. If the infimum of two $S$-probabilities $p$ and $q$ exists, we will denote the infimum by $p\wedge q$. Finally, we mention that if $p\le q$ for $p,q\in P$ then $q-p\in P$ and $q-p=q\wedge p'$ (cf.\ e.g.\ \cite{DDL}). In case of $p\le q$ and $p\ne q$ we write $p<q$.

The definition of algebras of $S$-probabilities is motivated by classical event fields, for which the pairwise orthogonality of a triple $A,B,C$ of events implies $A\subseteq B'\cap C'=(B\cup C)'$, which in terms of functions means $p\le1-(q+r)$.

We further mention that any algebra of numerical events is an orthomodular poset (with respect to $\le$ and $'$) which admits a full set of states and any orthomodular poset that admits a full set of states can be represented by an algebra of numerical events (cf.\ \cite{MT}). So, in particular, all Boolean logics and all the frequently used Hilbert logics are among the algebras of $S$-probabilities. (Concerning the question how Hilbert logics feature as algebras of $S$-probabilities see e.g.\ \cite{DLM}.)

Regarding notations one has to be aware that the term ``algebra of $S$-probabilities'' is somewhat misleading, because an algebra of $S$-probabilities $(P,\le,0,1,{}',+)$ is an algebra in the algebraic sense only in respect to the nullary operations $0$, $1$ and the unary operation $'$, but is a partial algebra what $+$ concerns. This is important when we speak of {\em subalgebras of algebras of numerical events}. Then we refer to the same $S$, the same partial order relation $\le$, the full nullary and unary operations and the partial operation $+$. In particular, a subalgebra of an algebra of $S$-probabilities which is a lattice (especially a Boolean algebra) need not be a lattice (Boolean algebra) any more. As usual, embedding an algebra $P$ of $S$-probabilities in an algebra $Q$ of $S$-probabilities means that $P$ is a subalgebra of $Q$ in the aforementioned sense, and embedding an $S$-probability $q\notin P$ together with $P$ into an algebra $Q$ of $S$-probabilities means $P\cup\{q\}\subseteq Q$ with $P$ a subalgebra, a fact that will be implicitly presumed when we speak of embedding $P\cup\{q\}$.

\begin{definition}\label{def2}
Let $P$ be an algebra of $S$-probabilities and $q$ a {\rm(}subsequently observed {\rm)} $S$-probability not belonging to $P$. If $q$ is added to $P$ and $P\cup\{q\}$ cannot be embedded into an algebra of $S$-probabilities, then $q$ will be called {\em destructive}. If $P$ is a Boolean algebra but there does not exist a Boolean algebra of $S$-probabilities into which $P\cup\{q\}$ can be embedded, then $q$ will be called {\em critical}. Conversely, if $P$ is not a Boolean algebra but $P\cup\{q\}$ can be embedded into a Boolean algebra, then $q$ will also be called {\em critical}. The changes to and from a Boolean algebra will be referred to as {\em critical changes}.
\end{definition}

We point out that the use of the word ``critical'' within Definition~\ref{def2} is only symmetric in respect to the assertions ``Boolean algebra, then no Boolean algebra'' and vice versa, but not as far as details are concerned. Starting from a Boolean algebra, $q$ is considered as critical if there does not exist a Boolean algebra into which $P\cup\{q\}$ can be be embedded irrespective of the fact that there might exist an other kind of algebra of $S$-probabilities into which $P\cup\{q\}$ could be embedded. If we start with a non-Boolean algebra of $S$-probabilities and call $q$ critical, this means that there always does exist a Boolean algebra of $S$-probabilities into which $P\cup\{q\}$ can be embedded. Obviously, in this case already $P$ can be embedded into a Boolean algebra of numerical events and $q$ is then to be considered as a cause for a reassessment reinforcing the information that one will probably deal with a classical process.

We further mention that within an algebras $P$ of $S$-probabilities any $p\in P\setminus\{0,1\}$ is {\em varying}, which means that it is neither $\le1/2$ nor $\ge1/2$ (cf.\ \cite{DDL}). ($p\le1/2$ would imply $p\perp p\perp p'\perp p$ and hence $p+p+p'\le1$, i.e.\ $p\le0$ which yields $p=0$. Dually, $p\ge1/2$ would imply $p'\perp p'\perp p\perp p'$ and hence $p'+p'+p\le1$, i.e. $p'\le0$ which yields $p'=0$, i.e.\ $p=1$.) An element different from $0$ and $1$ that is varying is also known as a {\em proper} element (cf.\ \cite{DL18} and \cite{DL21}). Moreover, the constant functions $0$ and $1$ are also considered as proper. So when adding a non-proper $S$-probability $q$ to an algebra of $S$-probabilities $P$, $q$ is destructive. In the following we will desist from this apparent situation and always assume that $S$-probabilities different from $0$ and $1$ that are added are varying.

\section{Illustrative examples}

\begin{example}
If one adds some $q$ to an algebra of $S$-probabilities $P$ and there exist $p_1,p_2\in P$ with $p_1\perp p_2$ and $q\perp p_1,p_2$ such that $p_1+p_2+q>1$ then $q$ is destructive. E.g.\ take $p_1=(3/4,1/4,1/4), p_2=(1/4,3/4,1/4), q=(1/4,1/4,3/4)$.
\end{example}

\begin{example}\label{ex1}
For $|S|= 4$ we consider the orthomodular lattice $P$ of length $2$ with the four atoms $p_1=(1,0,0,1)$, $p_2=(0,1,0,1)$, $p_1'=(0,1,1,0)$ and $p_2'=(1,0,1,0)$ which is isomorphic as a lattice to the lattice commonly known as $\mathbf{MO_2}$. $P$ is not a Boolean algebra, but $P$ together with some $q\notin P$ mapping $S$ to $\{0,1\}$ can be embedded in a Boolean algebra, namely the Boolean algebra consisting of all sixteen quadruples of $0$ and $1$. Therefore, adding $q$ does entail a critical change.
\end{example}

\begin{example}\label{ex2}
Given the four-element Boolean algebra $P:=\{0,p_1,p_1',1\}$ with $p_1,p_1'$ from Example~\ref{ex1} and $q=(1/4,0,0,1)$. With reference to $q\le p_1$ we consider the element $q+p_1'=(1/4,1,1,1)=:r$ and due to $r'\le p_1$ we set up the element $r'+p_1'$ which turns out to coincide with $q'$. This way we obtain that $0,p_1,p_1',q,q',r,r',1$ constitute an eight-element Boolean algebra which contains $P$ as a subalgebra, showing that $q$ is not critical. 
\end{example}

\begin{example}
Assume $p=(1,0,1/2,1/2)$ and $P$ to be the Boolean algebra consisting of $0,p,p',1$. Further, let $q=(1/2,1/2,0,1)$. Then there does exist an an algebra of $S$-prob\-a\-bil\-i\-ties into which $P\cup\{q\}$ can be embedded, namely the orthomodular lattice $\mathbf Q$ with the four pairwise not comparable elements $p,p',q,q'$, but not a Boolean algebra. Within a Boolean algebra of $S$-probabilities there would have to exist the infima of any two of these four elements which then turn out all to be $\le1/2$ as well as their pairwise suprema which are all $\ge1/2$. This means these elements would not be proper unless they were $0$ and $1$, respectively. Assuming that they are $0$ and $1$ there would have to exist a Boolean algebra containing a subalgebra {\rm(}in the algebraic sense{\rm)} isomorphic to $Q$, a contradiction. So $q$ is critical, but not destructive.
\end{example}

\begin{remark}
Let $P$ be an algebra of $S$-probabilities, $p\in P$ and $q\in[p-1/2,p)\cup(p,p+1/2]$. Then $q$ is destructive.
\end{remark}

\begin{proof}
Assume that $P\cup\{q\}$ can be embedded into an algebra of $S$-probabilities. \\
If $q\in[p-1/2,p)$ then $p'+q\in[1/2,1)$, a contradiction. \\
If $q\in(p,p+1/2]$ then $p+q'\in[1/2,1)$, a contradiction.
\end{proof}

\begin{example}
If $P=\{(0,0),(1/8,5/8),(7/8,3/8),(1,1)\}$ and $q=(3/8,7/8)$ then $q\in(p,p+1/2]$ for $p=(1/8,5/8)$ and hence $q$ is destructive.	
\end{example}

As for the question where the knowledge  comes from that a logic is appropriate for a certain situation, quite often the underlying logic is supposed due to previous experiments and has then to be verified by further experiments. It can also happen that there is no prior information and only a few numerical events  are known or gained by experiments, in particular if $S$ is finite or when the $S$-probabilities have only the values $0$ and $1$ (cf.\ \cite{D12}, \cite{DL18} and \cite{DLM}). Algebras of $S$-probabilities which have only the values $0$ and $1$ are so-called {\em concrete logics}, that are logics which have a set representation (cf.\ \cite P). In case of ascertained two-valued $S$-probabilities one can easily construct an appropriate algebra of numerical events by setting up the partial algebra generated by the given $S$-probabilities in respect to the operations $0$, $1$ and $'$ and the partial operation $+$ (cf.\ \cite{D12}). -- The only questionable axiom (c) of Definition \ref{def1} is apparently fulfilled because for any orthogonal triple $\{p,q,r\}$ at most one of the elements can be $1$. As for algebras of $S$-probabilities with arbitrary ranges we point out that there are several papers which deal with the question of characterizing them by specific internal connections of their elements (cf.\ \cite{DDL}, \cite{DL13} and \cite{DL14}). -- For a general overview on identifying quantum logics by numerical events we refer to the paper \cite{D20}.

In this paper we will focus on the fact how adding new $S$-probabilities to known algebras of $S$-probabilities might change the character of an algebra of $S$-probabilities in respect to being a Boolean algebra or not, which will make a total difference between the insight that physical matters take place along classical lines or follow laws of quantum mechanics. In particular we will study the situation when new states are added and how certain properties of elements in relation to a newly added numerical event might change the character of an algebra of $S$-probabilities. Moreover, we will examine special classes of $S$-probabilities for this end.

\section{Extending the set of states}

It sometimes happens that a certain algebra of $S$-probabilities $P=(P,\leq,{}',0,1,+)$ is known to underlie an experiment and a new numerical event $q$, the values of which are not only relevant for all $s\in S$ but also for some more states, pertaining to a further observable gains significance. To start with we assume that there is just one such additional state $\overline s$ and denote the set $S\cup\{\overline s\}$ by $\overline S$. Further, we suppose that for all $p\in P$ the probability values $p(\overline s)$ are $0$, i.e.\ we assume that the state $\overline s$ is not relevant for all $p\in P$. To cope with this new situation we will set up an algebra of $\overline S$-probabilities which has the same characteristic features as $P$. For this end we now define two kinds of new $\overline S$-probabilities which we will denote by $(p,0)$ and $(p,1)$, respectively, with the properties $(p,0)(s)=(p,1)(s)=p(s)$ for $s\in S$ and $(p,0)(\overline s)=0$ and $(p,1)(\overline s)=1$. The set of all these $\overline S$-probabilities will be denoted by $\overline P$. Next we extend the order $\le$ to $\overline P$. For $(p_1,v),(p_2,w)\in\overline P$ with $p_1,p_2\in P$ and $v,w\in\{0,1\}$ we define $(p_1,v)\le(p_2,w)$ if and only if $p_1\le p_2$ and $v\le w$. Further, we define $(p,v)^*=(p',v')$ with $v'=0$ if $v=1$ and $v'=1$ for $v=0$ and agree on the nullary operations $\overline0=(0,0)$ and $\overline1=(1,1)$. As for the partial operation $+$ in $\overline P$, $+$ should have the same meaning as in $P$. The transition from $(P,\le,0,1,{}',+)$ to $(\overline P,\le,\overline0,\overline1,^*,+)$ will be called a {\em $0,1$-extension} of $P$.

\begin{theorem}\label{th1}
$P$ is an algebra of $S$-probabilities if and only if $\overline P$ is an algebra of $\overline S$-probabilities. If $P$ is an algebra of $S$-probabilities then $P$ is a concrete logic, an orthomodular lattice or a Boolean algebra if and only if $\overline P$ has the same structure.
\end{theorem}

\begin{proof}
The poset $\mathbf P_2=(\{0,1\},\le)$ can be conceived as an algebra of $\{\overline s\}$-probabilities. The only requirement of Definition~\ref{def1} that has to be given some thought is axiom (c). But this is valid, because the only non-trivial orthogonal triple of $\mathbf P_2$ is $\{0,0,1\}$. This way $\overline P$ can be considered as the product of the poset $(P,\le)$ and of $\mathbf P_2$ which by Theorem~2.8 of \cite{DDL} is an algebra of $\overline S$-probabilities. That a $0,1$-extension of a concrete logic leads to a concrete logic is obvious. If $P$ was an orthomodular lattice or a Boolean algebra, also $\overline P$ is an orthomodular lattice or a Boolean algebra, because a $0,1$-extension means that $P$ is multiplied by an algebra within the same variety. 
\end{proof} 

Sticking to the above notation we now discuss adding a new element $\overline q=(q,c)\notin\overline P$ to $\overline P$ with $(q,c)(s)=q(s)$ for $s\in S$ and $(q,c)(\overline s) = c$, $c\in\mathbb R$, $0<c<1$.

\begin{theorem}\label{th2}
Let $\overline P$ be a Boolean algebra and $\overline q=(q,c)\notin\overline P$ an $\overline S$-probability to be added to $\overline P$ with $q$ not comparable to any element of $P\setminus\{0,1\}$. If there exists some $p\in P$ such that either $\min\{p(s),q(s)\}\le1/2$ for each $s\in S$ or $\max\{p(s),q(s)\}\ge1/2$ for each $s\in S$ then there does exist an algebra of numerical events into which $\overline P\cup\{\overline q\}$ can be embedded, but $\overline q$ is critical.
\end{theorem}

\begin{proof}
Let $p\in P$ be an $S$-probability for which $\min\{p(s),q(s)\}\le1/2$ for each $s\in S$. Then $\overline P\cup\{\overline q,\overline q^*\}$ constitutes an algebra of $\overline S$-probabilities. If there would exist a Boolean algebra $\mathbf Q$ in which $\overline P\cup\{\overline q,\overline q^*\}$ could be embedded, then any lower bound $b$ of $(p,0)$ and $(q,c)$ within $\mathbf Q$ would be $\le(1/2,0)$ and hence, due to the fact that $b$ must be proper, $b$, and in particular the infimum of $(p,0)$ and $(q,c)$ in $\mathbf Q$, would have to be $0$. But, $\mathbf Q$ being a Boolean algebra, $(p,0)\wedge(q,c)=0$ implies $(p,0)\le(q',c')$ from which we infer $p\le q'$, a contradiction to $q$ not being comparable to any element of $P$. If $\max\{p(s),q(s)\}\ge1/2$ dual arguments also yield a contradiction.
\end{proof}.

Adding an element $\overline q\notin\overline P$ to a $0,1$-extension $\overline P$ of an algebra of $S$-probabilities $P$ the values of which on $S$ are not relevant might be expressed by assuming that $\overline q=(0,c)$. However, this may be problematic.

\begin{theorem}\label{th3}
Let $P$ be an algebra of $S$-probabilities, $\overline P$ the $0,1$-extension of $P$ subject to the addition of a state $\overline s$ to $S$ and $\overline q\notin\overline P$. If the values of $\overline q$ on $S$ coincide with the values of some $p\in P$ and $\overline q(\overline s)\notin\{0,1\}$ then $\overline q$ is destructive. If $P$ is a concrete logic that is not a Boolean algebra and $\overline q(\overline s)\in\{0,1\}$ then $\overline q$ is critical.
\end{theorem}

\begin{proof}
Assume $\overline q=(p,c)$ with some $p\in P$. Then $(p,0)<(p,c)<(p,1)$ in $\overline P$ from which we can infer that $(p,0)\perp(p',c')$ and $(p,c)\perp(p',0)$. Adding these pairs of orthogonal elements we obtain that both $(1,c')$ and $(1,c)$ belong to $\overline P$. If $c\le1/2$ then $(1,c')$ is not varying, if $c\ge1/2$ then $(1,c)$ is not varying, hence $\overline q$ is destructive. Next we suppose that $P$ is a concrete logic that is not a Boolean algebra and $\overline q(\overline s)\in\{0,1\}$. Then, according to Theorem~\ref{th1}, also $\overline P$ is not a Boolean algebra. Since $\overline P\cup\{\overline q\}$ is a subposet of the Boolean algebra $2^{\overline S}$, $\overline q$ is critical.
\end{proof} 

\section{Special elements}

\begin{definition}\label{def3}
Let $p$ and $q$ be proper $S$-probabilities. $p$ and $q$ are called {\em partially reciprocal}, if either $\min\{p(s),q(s)\}\le1/2$ for all $s\in S$, or else $\max\{p(s),q(s)\}\ge1/2$ for all $s\in S$. In the first case we will call $p$ and $q$ {\em partially reciprocal below $1/2$} and in the second case {\em partially reciprocal above $1/2$}. If $p$ and $q$ are both, partially reciprocal below and above $1/2$, then we will call them {\em reciprocal}.
\end{definition}

The wording ``reciprocal'' is derived from the property that if $p(s)<1/2$ for an $s\in S$ then $q(s)\geq1/2$ and if $p(s)>1/2$ for an $s\in S$ then $q(s)\leq 1/2$, and the same being true with reversed roles of $p$ and $q$. Due to the fact that $p$ and $q$ are interchangeable we will also say that $p$ and $q$ are reciprocal to each other. 

\begin{lemma}\label{lm1}
If the elements $p$ and $q$ of an be algebra of $S$-probabilities $P$ are partially reciprocal below $1/2$ and $u\le p,q$ for some $u\in P$ then $u=0$. If $p$ and $q$ are partially reciprocal above $1/2$ and $v\ge p,q$ for some $v\in P$ then $v=1$.
\end{lemma}

\begin{proof}
If $p,q$ are partially reciprocal below $1/2$ and $u(s)\le p(s),q(s)$ for all $s\in S$, then $u(s)\le\min\{p(s),q(s)\}\le1/2$. This means that $u$ is not proper unless it is $0$. Dually we obtain that $v\ge p,q$ for elements $p,q$ which are partially reciprocal above $1/2$ implies $v=1$.
\end{proof}

\begin{theorem}\label{th4}
Let $P$ be a Boolean algebra of $S$-probabilities and $q$ a varying $S$-probability not belonging to $P$. Then $q$ is critical if there exists some $p\in P$ such that one of the following conditions hold:
\begin{enumerate}
\item[{\rm(i)}] $p$ and $q$ are partially reciprocal below $1/2$ and there exists some $s\in S$ with $p(s)+q(s)>1$,
\item[{\rm(ii)}] $p$ and $q$ are partially reciprocal above $1/2$ and there exists some $s\in S$ with $p(s)+q(s)<1$,
\item[{\rm(iii)}] $p$ and $q$ are reciprocal.
\end{enumerate} 
\end{theorem}

\begin{proof}
We assume that there exists a Boolean algebra $\mathbf Q$ of $S$-probabilities comprising $P\cup\{q\}$.
\begin{enumerate}
\item[(i)] If $p$ and $q$ are partially reciprocal below $1/2$ and there exists some $s\in S$ with $p(s)+q(s)>1$ then according to Lemma~\ref{lm1} $u\le p,q $ in $\mathbf Q$ entails $u=0$, i.e.\ $p\wedge q=0$ from which we obtain due to $\mathbf Q$ being a Boolean algebra that $p\le q'$ and hence $p+q\leq 1$, a contradiction to $p(s)+q(s)>1$ for some $s\in S$.
\item[(ii)] By arguments dual to the above ones we get $p'\wedge q'=0$ which yields $p'\leq q$, a contradiction to the existence of an $s\in S$ with $p'(s)>q(s)$ which follows from the assumption $p(s)+q(s)<1$.
\item[(iii)] By means of Lemma~\ref{lm1} we obtain that $p\wedge q=0$ and $p\vee q=1$ within $\mathbf Q$. Since $q\notin P$ this implies that $q$ and $p'$ are two different complements of $p$ which contradicts the uniqueness of the complement within a Boolean algebra.
\end{enumerate}
\end{proof}

\begin{theorem}\label{th5}
Let $P$ be an algebra of $S$-probabilities that is {\em not} a Boolean algebra and $q\notin P$. Then $q$ is {\em not} critical if
\begin{enumerate}
\item[{\rm(i)}] there exists some $p\in P$ such that $p,q,p',q'$ are pairwise different and pairwise reciprocal, and $q$ is arbitrarily chosen,
\item[{\rm(ii)}] there exist $p_1,p_2\in P$ such that $p_1,p_2,p_1',p_2'$ are pairwise different and pairwise reciprocal, and $q$ arbitrarily chosen,
\item[{\rm(iii)}] there exist $p_1,p_2,p_3,p_4\in P$ pairwise partially reciprocal below $1/2$ such that $p_1\perp p_2$, $p_3\perp p_4$ and $p_1+p_2=p_3+p_4$, and $q$ arbitrarily chosen,
\item[{\rm(iv)}] there exist $p_1,p_2,p_3\in P$ such that together with $p_4:=q$ case {\rm(iii)} applies.
\end{enumerate}
\end{theorem}

\begin{proof}
Assume $\mathbf Q$ to be a Boolean algebra of $S$-probabilities into which $P\cup\{q\}$ can be embedded.
\begin{enumerate}
\item[(i)] Lemma~\ref{lm1} implies that within $Q$ we have $p\vee q=p\vee q'=p'\vee q=p'\vee q'=p\vee p'=q\vee q'=1$ and $p\wedge q=p\wedge q'=p'\wedge q=p'\wedge q'=p\wedge p'=q\wedge q'=0$ from which we infer that $\mathbf{MO_2}$ (for its definition see above) is a subalgebra of $\mathbf Q$, a contradiction since $\mathbf Q$ is assumed to be a Boolean algebra.
\item[(ii)] We observe that if $P\cup\{q\}$ can be embedded into a Boolean algebra $\mathbf Q$, then already $P$ can be embedded into a Boolean algebra and within this Boolean algebra $p_1,p_2,p_1',p_2'$ take on the role of $p,q,p',q'$ in case (i).
\item[(iii)] Like before we can provide a proof within $P$ and assume $\mathbf Q$ to be a Boolean algebra of $S$-probabilities into which $P$ is embedded. By Lemma~\ref{lm1} $p_i\wedge p_j=0$ for $i\ne j$ and $i,j\in\{1,2,3,4\}$. Since $p_i,p_j$ for $i\ne j$ are partially reciprocal below $1/2$, $p_i',p_j'$ are partially reciprocal above $1/2$ and hence again by Lemma~\ref{lm1} $p_i'\vee p_j'=1$ for $i\ne j$ and $i,j\in\{1,2,3,4\}$. From the assumption $p_1+p_2=p_3+p_4=:e$ we can infer that $p_1\vee p_2=p_3\vee p_4=e$. We will further show that $e_1:=p_1\vee p_3$, $e_2:=p_1\vee p_4$, $e_3:=p_2\vee p_3$ and $e_4:=p_2\vee p_4$ are all equal to $e$. For this end we first we observe that $p_i\le e$ for $i=1,2,3,4$. Then we deduce
$e_1=e_1\vee(p_2\wedge p_4)=(p_1\vee p_3\vee p_2)\wedge(p_1\vee p_3\vee p_4)=(p_3\vee e)\wedge(p_1\vee e)=e$. \\
In the same manner considering for $i=2,3,4$ $e_i\wedge(p_j\wedge p_k)$ with $j\ne k$ and $\{j,k\}$ complementary to the two indices of the elements $p$ constituting $e_i$ we obtain $e_i=e$. It then follows that $p_i'\wedge p_j'=e'$ for $i\ne j$ and $i,j\in\{1,2,3,4\}$. Moreover, $p_1\leq p_2'$, $p_2\leq p_1'$, $p_3\leq p_4'$, $p_4\leq p_3'$, $e=p_1\vee p_2\vee p_3\vee p_4 \nleq p_i'$ for $i=1,2,3,4$ and $e'=p_1'\wedge p_2'\wedge p_3'\wedge p_4'\ngeq p_i$. This way we have come upon a subalgebra of $\mathbf Q$ which is isomorphic to $\mathbf{MO_2}\times\mathbf P_2$ ($\mathbf P_2$ denoting the two-element Boolean algebra), again a contradiction to $\mathbf Q$ being a Boolean algebra.
\item[(iv)] We assume $\mathbf Q$ to comprise $P\cup\{q\}$ and have the same arguments for a proof like in case (iii).
\end{enumerate}
\end{proof}

\section{Special algebras of numerical events}

\begin{lemma}\label{lem1}
Let $n\ge2$ and $P\subseteq[0,1]^S$. Then $P$ is an algebra of $S$-probabilities that is a $2^n$-element Boolean algebra if and only if there exist varying $S$-probabilities $p_1,\ldots,p_n$ satisfying $p_N=1$ and $P=\{p_I\mid I\subseteq N\}$ where $N:=\{1,\ldots,n\}$ and $p_I:=\sum\limits_{i\in I}p_i$ for all $I\subseteq N$. In this case the mapping $I\mapsto p_I$ from $2^N$ to $P$ is an isomorphism from the Boolean algebra $(2^N,\subseteq,{}',\emptyset,N)$ to $(P,\le,{}',0,1)$ where $I':=N\setminus I$ for all $I\subseteq N$. Moreover, $p_1,\ldots,p_n$ are the atoms of $(P,\le)$.
\end{lemma}

\begin{proof}
First assume $P$ to be an algebra of $S$-probabilities that	is a $2^n$-element Boolean algebra. Let $p_1,\ldots,p_n$ denote the atoms of $(P,\le)$. Then $p_1,\ldots,p_n$ are varying and pairwise orthogonal, $p_N=\bigvee\limits_{i\in N}p_i=1$ and $P=\{\bigvee\limits_{i\in I}p_i\mid I\subseteq N\}=\{p_I\mid I\subseteq N\}$. Conversely, assume $p_1,\ldots,p_n\in P$ to be varying and satisfying $p_N=1$ and $P=\{p_I\mid I\subseteq N\}$. Then conditions (a) and (b) of Definition~\ref{def1} are satisfied. Now let $I,J,K\subseteq N$. If $I\cap J=\emptyset$ then $p_I\perp p_J$. If, conversely, $p_I\perp p_J$ and $I\cap J\ne\emptyset$ then there exists some $i\in I\cap J$ and hence $2p_i\le p_I+p_J\le1$ contradicting the fact that $p_i$ is varying. Therefore $I\cap J=\emptyset$ if and only if $p_I\perp p_J$. If $p_I,p_J,p_K$ are pairwise orthogonal then $I,J,K$ are pairwise disjoint and hence $p_I+p_J+p_K=p_{I\cup J\cup k}\in P$ proving condition (c) of Definition~\ref{def1}. This shows that $P$ is an algebra of $S$-probabilities. We have $p_{I'}=1-p_I=(p_I)'$ and hence the following are equivalent: $I\subseteq J$, $I\cap J'=\emptyset$, $p_I\perp p_{J'}$, $p_I\perp(p_J)'$, $p_I\le p_J$. Finally, $p_\emptyset=0$ and $p_N=1$. Altogether we see that the mapping $I\mapsto p_I$ from $2^N$ to $P$ is an isomorphism from the $2^n$-element Boolean algebra $(2^N,\subseteq,{}',\emptyset,N)$ to $(P,\le,{}',0,1)$ and hence the latter is a $2^n$-element Boolean algebra, too.
\end{proof}

The following theorem generalizes Example~\ref{ex2}.

\begin{theorem}
Let $n\ge2$, let $p_1,\ldots,p_n$ be varying $S$-probabilities satisfying $p_N=1$ where $N:=\{1,\ldots,n\}$ and $p_I:=\sum\limits_{i\in I}p_i$ for all $I\subseteq N$, let $P$ denote the algebra $\{p_I\mid I\subseteq N\}$ of $S$-probabilities that is a $2^n$-element Boolean algebra and let $q$ be a varying $S$-probability satisfying $q<p_1$ such that $p_1-q$ is proper. Then $q\notin P$ and $q$ is not critical.
\end{theorem}

\begin{proof}
Put $q_0:=q$, $q_1:=p_1-q$ and $q_i:=p_i$ for all $i=2,\ldots,n$. Then $q_0,\ldots,q_n$ are varying $S$-probabilities satisfying $q_{N_0}=1$ where $N_0:=\{0,\ldots,n\}$ and $q_I:=\sum\limits_{i\in I}q_i$ for all $I\subseteq N_0$. According to Lemma~\ref{lem1}, $Q:=\{q_I\mid I\subseteq N_0\}$ is an algebra of $S$-probabilities that is a $2^{n+1}$-element Boolean algebra including $P\cup\{q\}$. Moreover, $P=\{q_I\mid I\subseteq N_0,\{0,1\}\subseteq I\mbox{ or }\{0,1\}\cap I=\emptyset\}$ and hence $q=q_{\{0\}}\notin P$.
\end{proof}

An algebra of $S$-probabilities that is isomorphic as a lattice to $\mathbf{MO_2}$ (or more general, to $\mathbf{MO_n}$ for $n>2$) is a logic that frequently turns up in connection with physical experiments, e.g.\ with experiments concerning the spin of electrons. ($\mathbf{MO_n}$ denotes the ortholattice of length two with $2n$ pairwise incomparable elements $p_1,p_2,\ldots,p_n,p_1',p_2',\ldots,p_n'$ strictly between $0$ and $1$.)

In the following, we will take a closer look at the algebra of $S$-probabilities isomorphic as a lattice to $\mathbf{MO_2}$. For short we will refer to this algebra of $S$-probabilities just as $\mathbf{MO_2}$.

As commonly used with posets we will write $a\parallel b$ if two elements $a,b$ of a poset  are not comparable.

\begin{lemma}\label{lm2}
Let $P$ be an algebra $\mathbf{MO_2}$ of $S$-probabilities, $g,h\in P$ with $g\parallel h$ and $|S|\le3$. Then there does not exist an $S$-probability $q\notin P$ such that $0<q<g<v<1$ and $0<q<h<v<1$ within a Boolean algebra of $S$-probabilities which includes $P\cup\{q\}$.
\end{lemma}

\begin{proof}
Suppose there exists some Boolean algebra of $S$-probabilities satisfying the mentioned property. \\
If $\min(g,h)\le1/2$ then $g\wedge h=0$ contradicting $q>0$. \\
If $\min(g,h')\le1/2$ then $g\wedge h'=0$ whence
\[
g\le h\vee g=(h\vee g)\wedge1=(h\vee g)\wedge(h\vee h')=h\vee(g\wedge h')=h
\]
contradicting $g\parallel h$. \\
Now we distinguish three cases. \\
1) $|S|=1$. \\
In this case there does nor exist an algebra of $S$-probabilities which is isomorphic to $\mathbf{MO_2}$. \\
2) $|S|=2$, say $S=\{s_1,s_2\}$. \\
W.l.o.g.\ we may assume $g(s_1)\le1/2$ and $g(s_2)\ge1/2$. \\
If $h(s_1)\le1/2$ and $h(s_2)\ge1/2$ then $\min(g,h')\le1/2$. \\
If $h(s_1)\ge1/2$ and $h(s_2)\le1/2$ then $\min(g,h)\le1/2$. \\
3) $|S|=2$, say $S=\{s_1,s_2,s_3\}$. \\
W.l.o.g.\ we may assume $g(s_1)\le1/2$, $g(s_2)\le1/2$ and $g(s_3)\ge1/2$. (In the case $g(s_1)\ge1/2$, $g(s_2)\ge1/2$ and $g(s_3)\le1/2$ the situation is dual.) \\
If $h(s_1)\le1/2$, $h(s_2)\le1/2$ and $h(s_3)\ge1/2$ then $\min(g,h')\le1/2$. \\
If $h(s_1)\le1/2$, $h(s_2)\ge1/2$ and $h(s_3)\le1/2$ then $\min(g,h)\le1/2$. \\
If $h(s_1)\le1/2$, $h(s_2)\ge1/2$ and $h(s_3)\ge1/2$ then $\min(g,h')\le1/2$. \\
If $h(s_1)\ge1/2$, $h(s_2)\le1/2$ and $h(s_3)\le1/2$ then $\min(g,h)\le1/2$. \\
If $h(s_1)\ge1/2$, $h(s_2)\le1/2$ and $h(s_3)\ge1/2$ then $\min(g,h')\le1/2$. \\
If $h(s_1)\ge1/2$, $h(s_2)\ge1/2$ and $h(s_3)\le1/2$ then $\min(g,h)\le1/2$.
\end{proof}

\begin{theorem}\label{th6}
Let $P$ be an algebra of $S$-probabilities being isomorphic to $\mathbf{MO_2}=\{0,p_1,$ $p_1',p_2,p_2',1\}$ and $q\notin P$ an $S$-probability with $0<q<p_1,p_2$ and $\mathbf Q$ a Boolean algebra of $S$-probabilities including $P\cup\{q\}$ such that $p_1\parallel p_2$ in $\mathbf Q$ and $p_1\vee p_2<1$. Then $|S|\ge4$.
\end{theorem}

\begin{proof}
in $\mathbf Q$ we have $0<q<p_1<p_1\vee p_2<1$ and $0<q<p_2<p_1\vee p_2<1$ which by Lemma \ref{lm2} can only occur in case $|S|\ge4$. 
\end{proof}

\begin{theorem}\label{th7}
Let $P$ be an algebra of $S$-probabilities isomorphic to $\mathbf{MO_2}=\{0,p_1,p_1',p_2,$ $p_2',1\}$ and $q$ a varying $S$-probability not belonging to $P$ and assume $p_1-q$, $p_2-q$ and $1+q-p_1-p_2$ to be varying $S$-probabilities. Then $P\cup\{q\}$ can be embedded into a Boolean algebra, i.e., $q$ is critical.
\end{theorem}

\begin{proof}
Since $q$, $p_1-q$, $p_2-q$ and $1+q-p_1-p_2$ are varying $S$-probabilities summing up to $1$ the proof follows from Lemma~\ref{lem1}.
\end{proof}

Authors' addresses:

Dietmar Dorninger \\
TU Wien \\
Faculty of Mathematics and Geoinformation \\
Institute of Discrete Mathematics and Geometry \\
Wiedner Hauptstra\ss e 8-10 \\
1040 Vienna \\
Austria

Helmut L\"anger \\
TU Wien \\
Faculty of Mathematics and Geoinformation \\
Institute of Discrete Mathematics and Geometry \\
Wiedner Hauptstra\ss e 8-10 \\
1040 Vienna \\
Austria, and \\
Palack\'y University Olomouc \\
Faculty of Science \\
Department of Algebra and Geometry \\
17.\ listopadu 12 \\
771 46 Olomouc \\
Czech Republic \\
helmut.laenger@tuwien.ac.at
\end{document}